\documentclass{article}
\usepackage{graphicx} 
\usepackage{amsmath, amsthm,amssymb}
\usepackage{hyperref}
\usepackage{cleveref}
\usepackage{amsthm}
\usepackage{tikz}
\usetikzlibrary{calc}
\usepackage{caption}
\newtheorem{theorem}{Theorem}

\newtheorem{open}{Open Problem}
\usepackage{verbatim}

\begin{document}

\title{Equilateral n-gons in planar integer lattices}
\author{Ghaura Mahabaduge\thanks{Department of Mathematics, Massachusetts Institute of Technology, Cambridge, MA 02139. Email: ghaura\_m@mit.edu.}}
\maketitle

\begin{abstract}
We study the existence of equilateral polygons in planar integer lattices. Maehara showed that it's sufficient to work with rectangular lattices $\Lambda(m) = L[(1,0),(0,\sqrt{m})]$ with $m \equiv 3 \pmod{4}$. Building on results of Maehara and of Iino and Sakiyama, we show that for every such $m$ there exists $N$ such that for all $n \geq N$, the lattice $\Lambda(m)$ contains an equilateral $n$-gon. This extends previous classifications of equilateral polygons in planar lattices.
\end{abstract}

\section{Introduction}

The study of equilateral polygons within lattices combines elements of geometry, number theory, and discrete mathematics. In particular, the problem of characterizing which polygons can be embedded in given planar lattices has drawn sustained interest. A central focus has been on determining for which values of $n$ there exists a convex equilateral $n$-gon whose vertices lie in a specified lattice.

Recent results by Maehara \cite{maehara2019planar} established that every planar integral lattice contains convex equilateral $n$-gons for all even $n \geq 4$, and for certain odd values of $n$, depending on a number-theoretic invariant of the lattice. In particular, Maehara showed that a planar integral lattice $L$ contains some equilateral polygon with an odd number of sides if and only if the square-free part of the square of the determinant of $L$, denoted $\nu(L)$, satisfies $\nu(L) \equiv 3 \pmod{4}$.

Building on this, Iino and Sakiyama \cite{iino2025planarlatticesequilateraloddgons} studied rectangular lattices of the form $\Lambda(m) = L[(1,0),(0,\sqrt{m})]$ for square-free integers $m \equiv 3 \pmod{4}$, and provided both necessary and sufficient conditions for the existence of equilateral polygons with a given odd number of sides. Among their key findings is that for such lattices, if a convex equilateral $n$-gon exists, then $n$ must be at least as large as every prime dividing $m$. They also proved that this is a sufficient condition when the largest prime factor of $m$ is less than 29 by a brute-force computer search.

In this paper, we extend the classification of equilateral polygons in rectangular lattices of the form $\Lambda(m)$, with $m \equiv 3 \pmod{4}$.

\begin{theorem}\label{thm: new theorem}
For any $m \equiv 3 \pmod{4}$ there exists $N$ such that for all $n \geq N$ there exists an equilateral polygon of size n.  
\end{theorem}

We will give an existence proof in \Cref{Section: Existence} followed by a explicit construction in \Cref{Section: Construction}.

\section{Background}

An equilateral polygon is a polygon whose sides are all equal in length. If every vertex of a polygon $P$ is in a set $S$, then we say $S$ contains the polygon $P$. For a square-free integer $m > 0$, we denote the rectangular lattice $L[(1,0),(0,\sqrt{m})]$ by $\Lambda(m)$. More generally for linearly independent vectors $\mathbf{a}, \mathbf{b}$ in the plane $\mathbb{R}^2$, $L[\mathbf{a}, \mathbf{b}]$ denotes the planar lattice generated by $\mathbf{a}$ and $\mathbf{b}$ (i.e. 
$L[\mathbf{a}, \mathbf{b}] = \{m\mathbf{a} + n\mathbf{b} \mid m, n \in \mathbb{Z}\} \subset \mathbb{R}^2$). If all the inner products of lattice vectors are integers, we call it a planar integer lattice. A lattice $L$ is said to be \textit{similar} to $L'$ if there exists a $\lambda > 0$ such that $\lambda L = \{\lambda \mathbf{x} \mid \mathbf{x} \in L\}$ is isometric to $L'$. For a planar lattice $L$, let $D(L)$ denote the area of a fundamental parallelogram of the lattice $L$. Thus, if $L = L[\mathbf{a}, \mathbf{b}]$, then $ D(L) = |\det(\mathbf{a}, \mathbf{b})|$. We will denote the square-free part of $D(L)^2$ by $\nu(L)$. 

Maehara \cite{maehara2019planar} proved that, a planar integral lattice $L$ is similar to a sublattice of $\Lambda(\nu)$, and $\Lambda(\nu)$ is similar to a sublattice of $L$, where $\nu$ is the square-free part of $D(L)^2$. Furthermore, it was proven by Maehara \cite{maehara2019planar} that If $\Lambda(m)$ contains an equilateral $n$-gon (resp. a convex equilateral $n$-gon), then it contains an equilateral $(n+2)$-gon (resp. a convex equilateral $(n+2)$-gon). Therefore, once we find some equilateral polygon of size $N$ where $N$ is odd, then we obtain equilateral $n$-gons for every odd $n$, such that $n\geq N$. 

Additionally, Maehara \cite{maehara2019planar} proved that a planar lattice $L$ contains a convex equilateral $n$-gon if and only if $L$ contains $n$ distinct vectors $e_1, e_2, \dots, e_n$ such that $|e_1| = |e_2| = \cdots = |e_n|$ and $e_1 + e_2 + \cdots + e_n = 0$. The proof will lead to the same statement for equilateral $n$-gons (which are not necessarily convex), when we remove the \textbf{distinct vectors} condition.

Therefore to prove \Cref{thm: new theorem}, it is sufficient to find $n$ vectors in $\Lambda(m)$ for any $m \equiv 3 \pmod{4}$, such that the vectors norms are equal to each other and the vector sum is zero.

\section{Existence proof of \Cref{thm: new theorem} }\label{Section: Existence}

\begin{proof}[Proof of Theorem~\ref{thm: new theorem}]
    Note that for even $n>2$ we can construct a parallelogram in the following manner. Let $v_1=(1,\sqrt{m})$ and $v_2=(1,-\sqrt{m})$. 
    A parallelogram can be created by first using  $(\frac{n}{2} -1)$ copies of $v_1$ and then 1 copy of $v_2$ and then $(\frac{n}{2} -1)$ copies of $-v_1$ followed by 1 copy of $-v_2$
    (example in \Cref{fig:parallelogram}).

    Now it's sufficient to prove that there exists some odd $N$ with a equilateral N-gon, because we can always extend to larger odd values by adding 2 edges.

    Let $w_0=(c,0)$,  $w_1=(a_1,b_1\sqrt{m})$, $w_2=(a_1,-b_1\sqrt{m})$, $w_3=(a_2,b_2\sqrt{m})$, $w_4=(a_2,-b_2\sqrt{m})$, where $c,a_1,a_2,b_1,b_2$ are integers. 

    We construct a closed walk using 1 copy of $w_0$, $|t_1|$ copies each of $w_1$ and $w_2$, and $|t_2|$ copies each of $w_3$ and $w_4$ with the signs chosen according to the sign of $t_i$. (Here $t_1$ and $t_2$ are integers)  

    To keep the size of all the vectors being equal, we require : $c^2-a_1^2=mb_1^2$ and $c^2-a_2^2=mb_2^2$.

    Note that by construction the sum vector of these $1+2(|t_1|+|t_2|)$ vectors lie on the x-axis. Therefore we need the condition $-c=2(a_1t_1+a_2t_2)$.

    Let's try to show that a solution exists to this system of equations. 

    Choose $c=\frac{m^2 + 5m + 4}{2}$, $b_1=m+2$, $a_1=\frac{m^2 + 3m + 4}{2}$,$b_2=m+4$, $a_2=\frac{m^2 + 3m - 4}{2}$.  (Since $4|(m+1)$ all of them are integers, furthermore $c$ is even.)

    This satisfies $c^2-a_1^2=mb_1^2$ and $c^2-a_2^2=mb_2^2$.

    If $\gcd(a_1,a_2) = 1$, then there exists $t_1,t_2$ integers such that $-c/2=(a_1t_1+a_2t_2)$ is satisfied.

    $\gcd(a_1,a_2) = \gcd(a_1,a_1 - 4) = \gcd(a_1,4)= \gcd(\frac{m^2 + 3m + 4}{2},4) $. Note that $\frac{m^2 + 3m + 4}{2}$ is odd. Hence $\gcd(a_1,a_2)=1$. This shows existence of a solution for $N=(1+ 2(|t_1|+|t_2|))$.

\end{proof}

\begin{figure}[h!]
    \centering
    \begin{tikzpicture}[scale=0.8,>=stealth]

    \coordinate (O) at (0,0);
    \coordinate (A) at (5,5);          
    \coordinate (B) at ($(A)+(1,-1)$); 
    \coordinate (C) at (1,-1);         

    \draw[->, thick] (O) -- (A) node[midway, left] {$v_1$};
    \draw[->, thick] (A) -- (B) node[midway, above right] {$v_2$};
    \draw[->, thick] (B) -- (C) node[midway, right] {$-v_1$};
    \draw[->, thick] (C) -- (O) node[midway, below left] {$-v_2$};

    \draw[dashed, gray] (O) -- (A) -- (B) -- (C) -- cycle;

    \end{tikzpicture}
    \caption{A parallelogram with vectors labeled $v_1$, $-v_1$, $v_2$, and $-v_2$. The longer sides have a length 5 times the shorter sides. $v_1$ is parallel to the reflection of $v_2$ over the $x$-axis.}
    \label{fig:parallelogram}
\end{figure}
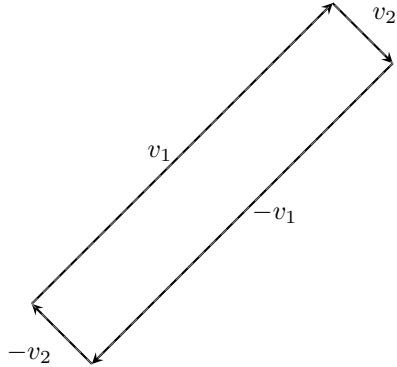

We gave the existence proof mentioned above because it shows how someone might try an approach to solve similar problems related to this problem. Now we give an explicit construction.

\section{Explicit Construction}\label{Section: Construction}

Let $m=4k-1$, $k=4k'+i$, where $i\in\{-1,0,1,2\}$ and let $\chi_i = k' -\frac{i(k-1)}{2}$. 
Then the choice 
\[
t_1= -\chi_i(16k'+4i+3),  \quad t_2 = -t_1 - i
\]
 satisfies $-c=2(a_1t_1+a_2t_2)$

We will now show why this works (Here $a_1,a_2,c$ are the same as in the proof of \Cref{thm: new theorem}).
We have
\begin{align*}
4(a_1t_1 + a_2t_2) &= 4(a_1t_1 + (a_1 -4)(-t_1-i))\\
&= 4 (4t_1 - ia_1 + 4i )\\
&=-4(-4t_1 + ia_1 - 4i )\\
&= -4(4\chi_i(16k'+4i+3 )+ i\frac{m^2 + 3m + 4}{2} - 4i )   \\
&=  -4(4\chi_i(16k'+4i+3 )+ i\frac{m^2 + 3m - 4}{2} ) \\
&= -4 ((4k' - 2i(k-1)) (16k' + 4i + 3 )\\
& \qquad \phantom{} + i (16k^2 -8k + 1 + 12k -3 - 4  )/2    ) \\
&= -4( (4k' - 2i(k-1)) (16k' + 4i + 3 ) + i (8k^2 + 2k - 3  ))  \\
&=-4( (4k' - 2i(k-1)) (4k + 3 ) + i (8k^2 + 2k - 3  )) \\
&=-4( (k - i - 2i(k-1)) (4k + 3 ) + i (8k^2 + 2k - 3  )) \\
&= -4( k(4k + 3 ) -i(2k-1)(4k+3) + i (8k^2 + 2k - 3  )) \\
& = -(4k)(4k+3)  \\
&= -(m+1)(m+4) \\
&= -(m^2 + 5m + 4)\\
&= -2c,
\end{align*}
and therefore $-c= 2(a_1t_1 + a_2t_2)$.

This gives an explicit construction for \Cref{thm: new theorem}.

\vspace{1cm}

\Cref{Illustration} shows an illustration of our construction for $m=7$. We have $c=44$, $a_1=37$, $a_2=33$, $b_1=9$, $b_2=11$, $t_1=11$ and $t_2=-13$. 

Which means we are using $11$ copies of $(37,9\sqrt{7})$, $11$ copies of $(37,-9\sqrt{7})$, $13$ copies of $(-33,11\sqrt{7})$, $13$ copies of $(-33,-11\sqrt{7})$, $1$ copy of $(44,0)$.

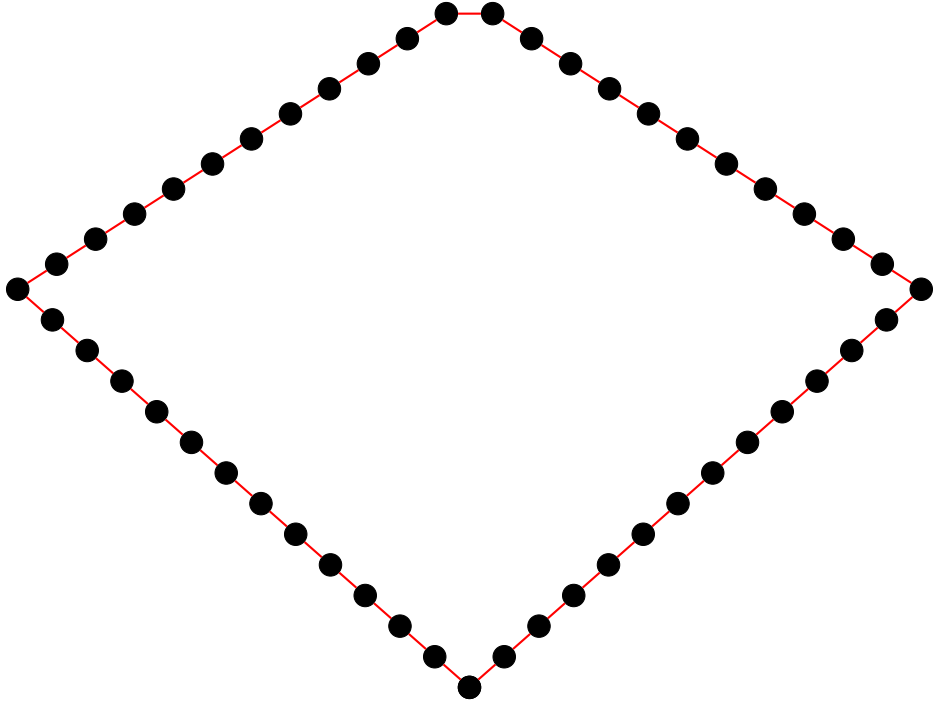
\begin{figure}[h!]
 \centering
  \begin{tikzpicture}[scale=0.14, mystyle/.style={circle, draw, fill=black, inner sep=3pt}]
  \node[mystyle] (A_1) at ({0},{0}) {};
  \node[mystyle] (A_2) at ({3.7},{0.9*sqrt(7)}) {};
  \node[mystyle] (A_3) at ({3.7*2},{0.9*sqrt(7)*2}) {};
  \node[mystyle] (A_4) at ({3.7*3},{0.9*sqrt(7)*3}) {};
  \node[mystyle] (A_5) at ({3.7*4},{0.9*sqrt(7)*4}) {};
  \node[mystyle] (A_6) at ({3.7*5},{0.9*sqrt(7)*5}) {};
  \node[mystyle] (A_7) at ({3.7*6},{0.9*sqrt(7)*6}) {};
  \node[mystyle] (A_8) at ({3.7*7},{0.9*sqrt(7)*7}) {};
  \node[mystyle] (A_9) at ({3.7*8},{0.9*sqrt(7)*8}) {};
  \node[mystyle] (A_10) at ({3.7*9},{0.9*sqrt(7)*9}) {};
  \node[mystyle] (A_11) at ({3.7*10},{0.9*sqrt(7)*10}) {};
  \node[mystyle] (A_12) at ({3.7*11},{0.9*sqrt(7)*11}) {};

  \node[mystyle] (B_1) at ({85.8-0},{0}) {};
  \node[mystyle] (B_2) at ({85.8-3.7},{0.9*sqrt(7)}) {};
  \node[mystyle] (B_3) at ({85.8-3.7*2},{0.9*sqrt(7)*2}) {};
  \node[mystyle] (B_4) at ({85.8-3.7*3},{0.9*sqrt(7)*3}) {};
  \node[mystyle] (B_5) at ({85.8-3.7*4},{0.9*sqrt(7)*4}) {};
  \node[mystyle] (B_6) at ({85.8-3.7*5},{0.9*sqrt(7)*5}) {};
  \node[mystyle] (B_7) at ({85.8-3.7*6},{0.9*sqrt(7)*6}) {};
  \node[mystyle] (B_8) at ({85.8-3.7*7},{0.9*sqrt(7)*7}) {};
  \node[mystyle] (B_9) at ({85.8-3.7*8},{0.9*sqrt(7)*8}) {};
  \node[mystyle] (B_10) at ({85.8-3.7*9},{0.9*sqrt(7)*9}) {};
  \node[mystyle] (B_11) at ({85.8-3.7*10},{0.9*sqrt(7)*10}) {};
  \node[mystyle] (B_12) at ({85.8-3.7*11},{0.9*sqrt(7)*11}) {};

  \node[mystyle] (C_2) at ({3.3},{-1.1*sqrt(7)}) {};
  \node[mystyle] (C_3) at ({3.3*2},{-1.1*sqrt(7)*2}) {};
  \node[mystyle] (C_4) at ({3.3*3},{-1.1*sqrt(7)*3}) {};
  \node[mystyle] (C_5) at ({3.3*4},{-1.1*sqrt(7)*4}) {};
  \node[mystyle] (C_6) at ({3.3*5},{-1.1*sqrt(7)*5}) {};
  \node[mystyle] (C_7) at ({3.3*6},{-1.1*sqrt(7)*6}) {};
  \node[mystyle] (C_8) at ({3.3*7},{-1.1*sqrt(7)*7}) {};
  \node[mystyle] (C_9) at ({3.3*8},{-1.1*sqrt(7)*8}) {};
  \node[mystyle] (C_10) at ({3.3*9},{-1.1*sqrt(7)*9}) {};
  \node[mystyle] (C_11) at ({3.3*10},{-1.1*sqrt(7)*10}) {};
  \node[mystyle] (C_12) at ({3.3*11},{-1.1*sqrt(7)*11}) {};
  \node[mystyle] (C_13) at ({3.3*12},{-1.1*sqrt(7)*12}) {};
  \node[mystyle] (C_14) at ({3.3*13},{-1.1*sqrt(7)*13}) {};

  \node[mystyle] (D_2) at ({85.8-3.3},{-1.1*sqrt(7)}) {};
  \node[mystyle] (D_3) at ({85.8-3.3*2},{-1.1*sqrt(7)*2}) {};
  \node[mystyle] (D_4) at ({85.8-3.3*3},{-1.1*sqrt(7)*3}) {};
  \node[mystyle] (D_5) at ({85.8-3.3*4},{-1.1*sqrt(7)*4}) {};
  \node[mystyle] (D_6) at ({85.8-3.3*5},{-1.1*sqrt(7)*5}) {};
  \node[mystyle] (D_7) at ({85.8-3.3*6},{-1.1*sqrt(7)*6}) {};
  \node[mystyle] (D_8) at ({85.8-3.3*7},{-1.1*sqrt(7)*7}) {};
  \node[mystyle] (D_9) at ({85.8-3.3*8},{-1.1*sqrt(7)*8}) {};
  \node[mystyle] (D_10) at ({85.8-3.3*9},{-1.1*sqrt(7)*9}) {};
  \node[mystyle] (D_11) at ({85.8-3.3*10},{-1.1*sqrt(7)*10}) {};
  \node[mystyle] (D_12) at ({85.8-3.3*11},{-1.1*sqrt(7)*11}) {};
  \node[mystyle] (D_13) at ({85.8-3.3*12},{-1.1*sqrt(7)*12}) {};
  \node[mystyle] (D_14) at ({85.8-3.3*13},{-1.1*sqrt(7)*13}) {};

  \draw[red,thick] (A_1) -- (A_2);
  \draw[red,thick] (A_2) -- (A_3);
  \draw[red,thick] (A_3) -- (A_4);
  \draw[red,thick] (A_4) -- (A_5);
  \draw[red,thick] (A_5) -- (A_6);
  \draw[red,thick] (A_6) -- (A_7);
  \draw[red,thick] (A_7) -- (A_8);
  \draw[red,thick] (A_8) -- (A_9);
  \draw[red,thick] (A_9) -- (A_10);
  \draw[red,thick] (A_10) -- (A_11);
  \draw[red,thick] (A_11) -- (A_12);

  \draw[red,thick] (B_1) -- (B_2);
  \draw[red,thick] (B_2) -- (B_3);
  \draw[red,thick] (B_3) -- (B_4);
  \draw[red,thick] (B_4) -- (B_5);
  \draw[red,thick] (B_5) -- (B_6);
  \draw[red,thick] (B_6) -- (B_7);
  \draw[red,thick] (B_7) -- (B_8);
  \draw[red,thick] (B_8) -- (B_9);
  \draw[red,thick] (B_9) -- (B_10);
  \draw[red,thick] (B_10) -- (B_11);
  \draw[red,thick] (B_11) -- (B_12);

  \draw[red,thick] (A_12) -- (B_12);

  \draw[red,thick] (A_1) -- (C_2);
  \draw[red,thick] (C_2) -- (C_3);
  \draw[red,thick] (C_3) -- (C_4);
  \draw[red,thick] (C_4) -- (C_5);
  \draw[red,thick] (C_5) -- (C_6);
  \draw[red,thick] (C_6) -- (C_7);
  \draw[red,thick] (C_7) -- (C_8);
  \draw[red,thick] (C_8) -- (C_9);
  \draw[red,thick] (C_9) -- (C_10);
  \draw[red,thick] (C_10) -- (C_11);
  \draw[red,thick] (C_11) -- (C_12);
  \draw[red,thick] (C_12) -- (C_13);
  \draw[red,thick] (C_13) -- (C_14);

  \draw[red,thick] (B_1) -- (D_2);
  \draw[red,thick] (D_2) -- (D_3);
  \draw[red,thick] (D_3) -- (D_4);
  \draw[red,thick] (D_4) -- (D_5);
  \draw[red,thick] (D_5) -- (D_6);
  \draw[red,thick] (D_6) -- (D_7);
  \draw[red,thick] (D_7) -- (D_8);
  \draw[red,thick] (D_8) -- (D_9);
  \draw[red,thick] (D_9) -- (D_10);
  \draw[red,thick] (D_10) -- (D_11);
  \draw[red,thick] (D_11) -- (D_12);
  \draw[red,thick] (D_12) -- (D_13);
  \draw[red,thick] (D_13) -- (D_14);

  \end{tikzpicture}
  \caption{An illustration for $m=7$}
  \label{Illustration}
 \end{figure}

\section{Open Problems}

A natural question to ask would be, whether it's possible to prove our main result with convex equilateral polygons. 

\begin{open}
    
For any $m \equiv 3 \pmod{4}$ does there exist $N$ such that for all $n \geq N$ there exists a \textbf{convex} equilateral polygon of size n? 

\end{open}

For the classification to be complete, one might ask whether we can make a construction to match the necessary condition given by Iino and Sakiyama \cite{iino2025planarlatticesequilateraloddgons} for every $m \equiv 3 \pmod{4}$.

\begin{open} 
For any $m \equiv 3 \pmod{4}$ does there exist an equilateral polygon of size $n'$, where $n'$ is the largest prime factor of m?
\end{open}

\begin{open}
    
For any $m \equiv 3 \pmod{4}$ does there exist a \textbf{convex} equilateral polygon of size $n'$, where $n'$ is the largest prime factor of m?

\end{open}

\vspace{1cm}

{\noindent \textbf{Acknowledgments.}} This research was conducted as part of a course project under the supervision of Prof. Henry Cohn. The author thanks Prof. Cohn for his guidance throughout the project.

\bibliographystyle{amsplain.bst}
\bibliography{Ref}

\end{document}